\documentclass[11pt,letterpaper,english]{amsart}
\usepackage[T1]{fontenc}
\usepackage[latin9]{inputenc}
\usepackage{babel}
\usepackage{prettyref}
\usepackage{mathrsfs}
\usepackage{mathtools}
\usepackage{amstext}
\usepackage{amsthm}
\usepackage{amssymb}
\usepackage{color}
\usepackage[unicode=true,
 bookmarks=false,
 breaklinks=false,pdfborder={0 0 1},backref=false,colorlinks=false]
 {hyperref}

\makeatletter

\numberwithin{equation}{section}
\numberwithin{figure}{section}
\theoremstyle{plain}
\newtheorem{thm}{\protect\theoremname}
\theoremstyle{plain}
\newtheorem{cor}[thm]{\protect\corollaryname}
\theoremstyle{plain}
\newtheorem{prop}[thm]{\protect\propositionname}
\theoremstyle{remark}
\newtheorem{rem}[thm]{\protect\remarkname}
\newtheorem*{rem*}{Remark}
\theoremstyle{definition}
\newtheorem{defn}[thm]{\protect\definitionname}
\theoremstyle{plain}
\newtheorem{lem}[thm]{\protect\lemmaname}

\usepackage{babel}

\@namedef{subjclassname@2020}{%
  \textup{2020} Mathematics Subject Classification}

\providecommand{\definitionname}{Definition}
\providecommand{\lemmaname}{Lemma}
\providecommand{\corollaryname}{Corollary}
\providecommand{\propositionname}{Proposition}
\providecommand{\remarkname}{Remark}
\providecommand{\theoremname}{Theorem}

\makeatother

\providecommand{\corollaryname}{Corollary}
\providecommand{\definitionname}{Definition}
\providecommand{\lemmaname}{Lemma}
\providecommand{\propositionname}{Proposition}
\providecommand{\remarkname}{Remark}
\providecommand{\theoremname}{Theorem}

\begin{document}
\title[KE Bergman metrics]{ Kähler-Einstein Bergman metrics on pseudoconvex domains of dimension
two}
\author{Nikhil Savale \& Ming Xiao}
\thanks{N. S. was partially supported by the DFG funded project CRC/TRR 191.
\\
 M. X. was partially supported by the NSF grants DMS-1800549 and DMS-2045104.}
\address{Universität zu Köln, Mathematisches Institut, Weyertal 86-90, 50931
Köln, Germany}
\email{nsavale@math.uni-koeln.de}
\address{Department of Mathematics, University of California at San Diego,
La Jolla, CA 92093, USA}
\email{m3xiao@ucsd.edu}
\subjclass[2000]{32F45, 32Q20, 81Q20 }
\begin{abstract}
We prove that a two dimensional pseudoconvex domain of finite type
with a Kähler-Einstein Bergman metric is biholomorphic to the unit
ball. This answers an old question of Yau for such domains. The proof
relies on asymptotics of derivatives of the Bergman kernel along critically
tangent paths approaching the boundary, where the order of tangency
equals the type of the boundary point being approached. 
\end{abstract}

\maketitle

\section{\label{sec:Introdunction} Introduction}

\global\long\def\-{\overline{}}%
\global\long\def\Ol{\overline{}}%
\global\long\def\k{K\"ahler}%
\global\long\def\ke{K\"ahler-Einstein}%
Let $D\subset\mathbb{C}^{n}$ be a bounded pseudoconvex domain. There
exist two natural canonical metrics defined in its interior. The first
is the Bergman metric \cite{Bergman70-book} defined using the Bergman
kernel. The other is the complete Kähler-Einstein metric in $D$,
whose existence was established by the work of Cheng-Yau and Mok-Yau
\cite{Cheng-Yau-80,Mok-Yau1983}. The importance of the metrics stems
from their biholomorphic invariance property and intimate connections
with the boundary geometry.


It is hence a natural question to ask when the two canonical metrics
coincide; i.e. when the Bergman metric on the domain $D$ is Kähler-Einstein.
It was asked, in some form by Yau \cite[pg. 679]{Yau82-problems},
whether this happens if and only if $D$ is homogeneous. The reverse
direction of Yau's question (i.e. if $D$ is homogeneous, then the
Bergman metric is Kähler-Einstein) follows from a simple observation
using the Bergman invariant function (cf. Fu-Wong \cite{Fu-Wong97}).
The challenging aspect of Yau's question is the forward direction
which is still wide open in its full generality. It should be noted
that homogeneous domains have been classified in \cite{Vinberg-Gindikin-PS-1963}
and the only smoothly bounded homogeneous domain is the ball, as a
consequence of Wong \cite{Wong1977} and Rosay \cite{Rosay1979}.

A more tractable case of Yau's question is when $D$ has strongly
pseudoconvex smooth boundary. An explicit conjecture in this case
was posed earlier by Cheng \cite{Ch}: if the Bergman metric of a
smoothly bounded strongly pseudoconvex domain is Kähler-Einstein,
then the domain is biholomorphic to the unit ball. The first breakthrough toward solving Cheng's conjecture is the seminal work \cite{Fu-Wong97}. In this paper, Fu and Wong brought in Fefferman's expansion of the Bergman kernel as well as its connection with the Ramadanov conjecture to study the problem, influencing later studies. Cheng's conjecture in $\mathbb{C}^2$ then follows by combining the results of Fu and Wong \cite{Fu-Wong97} with a uniformization theorem established by Nemirovski and Shafikov \cite{Nemirovski-Shafikov-2006}. In higher dimensions, the conjecture was proved more recently by Huang and
the second author \cite{Huang-Xiao21}.
See also \cite{HuLi,EXX}
for some
further work on generalizing Cheng's conjecture to Stein spaces. 

The proofs of Cheng's conjecture in \cite{Fu-Wong97,Huang-Xiao21}
fundamentally use Fefferman's asymptotic result \cite{Fefferman74}
for the Bergman kernel, together with its connections to the CR invariant
theory for the boundary geometry. In the broader context of pseudoconvex
finite type domains, both tools are either absent or insufficiently
understood. As a result, little progress was made towards understanding
Yau's question in this context. To the best knowledge of the authors,
the only known result was due to Fu-Wong \cite{Fu-Wong97}. They showed
that, on a smoothly bounded, complete Reinhardt, pseudoconvex domain
of finite type domain in $\mathbb{C}^{2}$, if the Bergman metric
is Kähler-Einstein, then the domain is biholomorphic to the unit ball.
Their proof utilized the non-tangential limit of the Bergman invariant
function (see Fu \cite{Fu1996}). Besides, their proof used the aid
of a computer, again reflecting the intricacy of the problem in the
more general finite type case.

Our main theorem below gives an affirmative answer to Yau's question
for smoothly bounded pseudoconvex domains of finite type in dimension
two.

\begin{thm}
\label{thm:main theorem} Let $D\subset\mathbb{C}^{2}$ be a smoothly
bounded pseudoconvex domain of finite type. If the Bergman metric
of $D$ is Kähler-Einstein, then $D$ is biholomorphic to the unit
ball in $\mathbb{C}^{2}$. 
\end{thm}

A key role is again played by the boundary asymptotics for the Bergman
kernel. For two dimensional pseudoconvex domains of finite type, Hsiao
and the first author \cite{HsiaoSavale-2022} recently described the
asymptotics of the Bergman kernel along transversal paths approaching
the boundary. For our proof we shall need to extend this asymptotic
result to tangential paths approaching the boundary. The paths shall further be chosen to be \textit{critically
tangent}; their order of tangency with the boundary equals the type
of the boundary point that is being approached (see Remark
\prettyref{rem:(Critical-tangency)} below for a further discussion
of this choice).

As a consequence of our main theorem, we also positively answer Yau's
question for two dimensional bounded domains with real analytic boundary
(such domains are always of finite type). 
\begin{cor}
Let $D\subset\mathbb{C}^{2}$ be a bounded pseudoconvex domain with
real analytic boundary. If the Bergman metric of $D$ is Kähler-Einstein,
then $D$ is biholomorphic to the unit ball in $\mathbb{C}^{2}$. 
\end{cor}

The article is organized as follows. We begin with some preliminaries
on the Bergman and Kähler-Einstein metrics in \prettyref{sec:Preliminaries}.
In \prettyref{sec:Bergman kernel der asymptotics}, we establish the
asymptotics for the Bergman kernel and its derivatives along a critically
tangent path. The leading term of the asymptotics is computed as well
in terms of a model Bergman kernel on the complex plane. Then we carry
out the requisite analysis of the model in \prettyref{sec: Analysis of the model}.
Finally we prove \prettyref{thm:main theorem} in \prettyref{sec:Main theorem proof}.

\medskip

\section{\label{sec:Preliminaries} Preliminaries}

In this section we begin with some requisite preliminaries on the
Bergman and Kähler-Einstein metrics.

Let $D\subset\mathbb{C}^{n}$ be a smoothly bounded domain. A boundary
defining function is a smooth function $\rho\in C^{\infty}\left(\bar{D}\right)$
satisfying $D=\left\{ \rho\left(z\right)<0\right\} \subset\mathbb{C}^{2}\textrm{ and }\left.d\rho\right|_{\partial D}\neq0.$
The CR and Levi-distributions on the boundary $X\coloneqq\partial D$
are defined via $T^{1,0}X=T^{1,0}\mathbb{C}^{2}\cap T_{\mathbb{C}}X$
and $HX\coloneqq\textrm{Re}\left[T^{1,0}X\oplus T^{0,1}X\right]$
respectively. The Levi form on the boundary is defined by 
\begin{align}
\mathscr{L} & \in\left(T^{1,0}X\right)^{*}\otimes\left(T^{0,1}X\right)^{*}\nonumber \\
\mathscr{L}\left(U,\bar{V}\right) & \coloneqq\partial\bar{\partial}\rho\left(U,\bar{V}\right)=-\overline{\partial}\rho\left(\left[U,\bar{V}\right]\right)\label{eq:Levi form}
\end{align}
for $U,V\in T^{1,0}X$. The domain is called \textit{strongly pseudoconvex}
if the Levi form is positive definite; and \textit{weakly pseudoconvex}
(or simply \textit{pseudoconvex}) if the Levi form is semi-definite.

We now recall the notion of finite type. There are two standard notions
of finite type (D'Angelo and Kohn/Bloom-Graham) of a smooth real hypersurface
$M$, and these happen to coincide in $\mathbb{C}^{2}.$ (The reader
is referred to \cite[page 252]{Baouendi-Ebenfelt-Rothschild-99} for more details).
The domain is called of \textit{finite type} (in the sense of Kohn/Bloom-Graham)
if the Levi-distribution $HX$ is bracket generating: $C^{\infty}\left(HX\right)$
generates $TX$ under the Lie bracket. In particular the \textit{type
of a point} on the boundary $x\in X=\partial D$ is the smallest integer
$r\left(x\right)$ such that $H_{x}X_{r\left(x\right)}=T_{x}X$, where
the subspaces $HX_{j}\subset TX$, $j=1,\ldots$ are inductively defined
by 
\begin{align}
HX_{1} & \coloneqq HX\nonumber \\
HX_{j+1} & \coloneqq HX+\left[HX_{j},HX\right],\quad\forall j\geq1.\label{eq:canonical flag}
\end{align}
In general, the function $x\mapsto r\left(x\right)$ is only upper
semi-continuous. The finite type hypothesis is then equivalent to
$r\coloneqq\max_{x\in X}r\left(x\right)<\infty.$ Note that the type
of a strongly pseudoconvex point $x$ is $r\left(x\right)=2$.

The Bergman projector of $D$ is the orthogonal projector 
\begin{equation}
K_{D}:L^{2}\left(D\right)\rightarrow L^{2}\left(D\right)\cap\mathcal{O}\left(D\right)\label{eq:Bergman kernel}
\end{equation}
from square integrable functions onto the closed subspace of square-integrable
holomorphic ones. Its Schwartz kernel, still denoted by $K_{D}\left(z,z'\right)\in L^{2}\left(D\times D\right),$
is called the Bergman kernel of $D$. It is well-known to be smooth
in the interior and positive along the diagonal. The Bergman metric
is the Kähler metric in the interior defined by 
\[
g_{\alpha\bar{\beta}}^{D}\coloneqq\partial_{\alpha}\partial_{\bar{\beta}}\ln K_{D}\left(z,z\right).
\]
Denote by $G=\det\left(g_{\alpha\bar{\beta}}^{D}\right)$ the determinant
of the above metric. The Ricci tensor of $g^{D}$ is by definition
$R_{\alpha\bar{\beta}}=-\partial_{\alpha}\partial_{\bar{\beta}}\ln G$.
The Bergman metric is always Kähler, and is further said to be \textit{Kähler-Einstein}
if $R_{\alpha\bar{\beta}}=cg_{\alpha\bar{\beta}}^{D}$ for some constant
$c$. Since $D$ is a bounded domain, the sign of $c$ must necessarily
be negative (cf. \cite[page 518]{Cheng-Yau-80}). The Bergman invariant
function is defined by $B\left(z\right)\coloneqq\frac{G\left(z\right)}{K_{D}\left(z,z\right)}$.
It follows from the transformation formula of the Bergman kernel that
the Bergman invariant function is invariant under biholomorphisms.

Next we briefly discuss the Kähler-Einstein metric. Recall the existence
of a complete Kähler-Einstein metric on $D\subset\mathbb{C}^{n}$
is governed by the following Dirichlet problem: 
\begin{align}
J\left(u\right)\coloneqq(-1)^{n}\det\begin{pmatrix}u & u_{\bar{\beta}}\\
u_{\alpha} & u_{\alpha\bar{\beta}}
\end{pmatrix} & =1\quad\textrm{in }D,\nonumber \\
u & =0\quad\text{on }\partial D\label{eq:Fefferman J equation}
\end{align}
with $u>0$ in $D$. Here $u_{\alpha}$ denotes $\partial_{z_{\alpha}}u$,
and likewise for $u_{\bar{\beta}}$ and $u_{\alpha\bar{\beta}}$.
The problem was first studied by Fefferman \cite{Fefferman74}, and
$J(\cdot)$ is often referred as Fefferman's complex Monge-Ampère
operator. Cheng and Yau \cite{Cheng-Yau-80} proved the existence
and uniqueness of an exact solution $u\in C^{\infty}(D)$ to \eqref{eq:Fefferman J equation},
on a smoothly bounded strongly pseudoconvex domain $D$. The function
$u$ is called the Cheng--Yau solution; and $-\partial\-{\partial}\log u$
gives rise to a complete Kähler-Einstein metric on $D$. Mok-Yau \cite{Mok-Yau1983}
further showed a bounded domain admits a complete Kähler-Einstein
metric if and only if it is a domain of holomorphy.

We next make some observations on the Monge-Ampère operator for later
applications. The left hand side of the first equation in \eqref{eq:Fefferman J equation}
can further be invariantly written as $J\left(u\right)=u^{n+1}\det\left[\partial\bar{\partial}\left(-\ln u\right)\right]$.
It may thus be computed in terms of any orthonormal frame $\{Z_{\alpha}\}_{\alpha=1}^{n}$
of $T^{1,0}\mathbb{C}^{n}$ as 
\begin{equation}
J\left(u\right)=\det\begin{pmatrix}u & \bar{Z}_{\beta}u\\
Z_{\alpha}u & Z_{\alpha}\bar{Z}_{\beta}u-\left[Z_{\alpha},\bar{Z}_{\beta}\right]^{0,1}u
\end{pmatrix}.\label{eq:Monge Ampere determinant}
\end{equation}
This can be proved using the identity 
\begin{equation}
\partial\bar{\partial}f\left(Z_{\alpha},\bar{Z}_{\beta}\right)=Z_{\alpha}\bar{Z}_{\beta}\left(f\right)-\bar{\partial}f\left(\left[Z_{\alpha},\bar{Z}_{\beta}\right]\right).\label{eq:ddbar calc.}
\end{equation}
Here the normality of $\{Z_{\alpha}\}_{\alpha=1}^{n}$ means each
of them has the same norm as $\partial_{z_{1}},\cdots,\partial_{z_{n}}$ with respect to the Euclidean metric.


The following proposition gives an equivalent condition for the Bergman
metric being Kähler-Einstein, which is easier to work with. The proof
is similar to \cite[Proposition 1.1]{Fu-Wong97} and \cite[Proposition 3.3]{HuLi}. 
\begin{prop}
\label{cor:Cor KE iff B fn const} Let $D\subset\mathbb{C}^{n},n\geq2,$
be a smoothly bounded pseudoconvex domain. Then its Bergman metric
$g^{D}$ is Kähler-Einstein if and only if the Bergman invariant function
is constant $B\left(z\right)\equiv\left(n+1\right)^{n}\frac{\pi^{n}}{n!}$.
This is also equivalent to the Bergman kernel $K_{D}$ satisfying
$J\left(K_{D}\right)=(-1)^{n}\frac{(n+1)^{n}\pi^{n}}{n!}K_{D}^{n+2}$. 
\end{prop}

\begin{proof}
We start with the proof of the first assertion. Since the reverse
direction is trivial, we only need to prove the forward part. Assume
the Bergman metric of $D$ is Kähler-Einstein.

Recall a smoothly bounded domain in $\mathbb{C}^{n}$ always has a
strongly pseudoconvex boundary point. Therefore we can find a strongly
pseudoconvex open connected piece $M$ of $\partial D$. Fix $p\in M$.
Next pick a small smoothly bounded strongly pseudoconvex domain $D'\subseteq D$
such that $D'\cap O=D\cap O$ and $\partial D'\cap O=\partial D\cap O=:M_{0}\subseteq M$
for some small ball $O$ in $\mathbb{C}^{n}$ centered at $p$. 

Write $K_{D'}$ for the Bergman kernel of $D'$. Then by the localization
of the Bergman kernel on pseudoconvex domains at a strongly pseudoconvex
boundary point (cf. Theorem 4.2 in Engliš \cite{En01}), 
there is a smooth function $\Phi$ in a neighborhood of $D'\cup M_{0}$
such that 
\begin{equation}
K_{D}=K_{D'}+\Phi~\text{on}~D'.\label{eq:eqnphi}
\end{equation}
Note that $K_{D'}$ obeys Fefferman asymptotic expansion on $D'$
by \cite{Fefferman74}. Combining this with \prettyref{eq:eqnphi},
we see for any defining function $\rho$ of $D\cap O$ with $D\cap O=\{z\in O:\rho(z)<0\}$,
the Bergman kernel $K_{D}$ also has the Fefferman type expansion
in $D\cap O:$ 
\begin{equation}
K_{D}=\frac{\phi}{\rho^{n+1}}+\psi\log(-\rho)\quad\text{on}~D\cap O.\label{eq:eqnkz}
\end{equation}
Here $\phi$ and $\psi$ are smooth in a neighborhood of $D'\cup M_{0}$
with $\phi$ nowhere zero on $M_{0}.$

Then by \prettyref{eq:eqnkz} and (the proof of) Theorem 1 of Klembeck
\cite{Kl78}, the Bergman metric of $D$ is asymptotically of constant
holomorphic sectional curvature $\frac{-2}{n+1}$ as $z\in D\rightarrow M_{0}$.
Consequently, the Bergman metric of $D$ is asymptotically of constant
Ricci curvature $-1$ as $z\in D\rightarrow M_{0}$ (To prove the
latter fact, alternatively one can apply a similar argument as page
510 of Cheng-Yau \cite{Cheng-Yau-80}). Therefore by the Kähler-Einstein
assumption, we must have $R_{i\Ol{j}}=-g_{i\Ol{j}}.$ This yields
$\partial\Ol\bar{\partial}\log B\equiv0$ in $D$. That is, $\log B$
is pluriharmonic in $D.$

Furthermore, by \prettyref{eq:eqnphi} and a similar argument as in
the proof of Lemma 3.2 in \cite{HuLi}, we have $B(z)\rightarrow\frac{(n+1)^{n}\pi^{n}}{n!}$
as $z\rightarrow M_{0}.$ Now write $\Delta=\{z\in\mathbb{C}:|z|<1\}$
for the unit disk. Let $f:\Delta\rightarrow O$ be an analytic disk
attached to $M_{0}.$ That is, $f$ is holomorphic in $\Delta$ and
continuous in $\Ol{\Delta}$ with $f(\Delta)\subset O\cap D$ and
$f(\partial\Delta)\subset M_{0}.$ Then $\log B(f)$ is harmonic in
$\Delta,$ continuous up to $\partial\Delta,$ and takes constant
value $\log\frac{(n+1)^{n}\pi^{n}}{n!}$ on $\partial\Delta.$ This
implies $B$ takes the constant value $\frac{(n+1)^{n}\pi^{n}}{n!}$
on $f(\Delta).$ But since $M_{0}$ is strongly pseudoconvex, we can
find a family $\mathcal{F}$ of analytic disks such that $\cup_{f\in\mathcal{F}}f(\Delta)$
fills up an open subset $U$ of $O\cap D$(cf. \cite[Chapter 8]{Baouendi-Ebenfelt-Rothschild-99}).
Thus $B$ is constant on $U$. Since $B$ is real analytic and $D$
is connected, we see $B\equiv\frac{(n+1)^{n}\pi^{n}}{n!}$.


Finally, note the formula $J(u)=u^{n+1}\det\left(\partial\overline{\partial}(-\ln u)\right)$
yields that, $B\left(z\right)=c$ if and only if $J\left(K_{D}\right)=(-1)^{n}cK_{D}^{n+2}$.
Then the second assertion of the proposition follows immediately. 
\end{proof}

As one anonymous referee pointed out, the first assertion in Proposition \prettyref{cor:Cor KE iff B fn const} also follows from an easier rescaling argument
as in Rosay \cite{Rosay1979} and Kim-Yu \cite{KY96}. This would allow one to relax the $C^{\infty}$ boundary regularity assumption to $C^2-$smoothness in Proposition \prettyref{cor:Cor KE iff B fn const}. Of course we still need the $C^{\infty}$ boundary regularity assumption in our later argument.

\section{\label{sec:Bergman kernel der asymptotics} The Bergman kernel and
its derivatives}

To prove \prettyref{thm:main theorem}, we shall fundamentally use
the asymptotics of the Bergman kernel on pseudoconvex domains of finite
type. In this section, we first briefly recall some classical and
recent known work, and then prove new results for asymptotics of the
Bergman kernel.

In \prettyref{sec:Preliminaries}, we already made use of Fefferman's
Bergman kernel asymptotics in the strongly pseudoconvex case. Let
$D$ be a strongly pseudoconvex domain with a defining function $\rho\in C^{\infty}\left(\bar{D}\right)$.
Fefferman \cite{Fefferman74} showed that the Bergman kernel of the
domain $D$ has an asymptotic expansion 
\begin{equation}
K_{D}\left(z,z\right)=a\left(z\right)\rho{}^{-n-1}+b\left(z\right)\ln\left(-\rho\right)\label{eq:Fefferman expansion}
\end{equation}
for some functions $a\left(z\right),b\left(z\right)\in C^{\infty}\left(\bar{D}\right)$.

Recently, the asymptotics in \prettyref{eq:Fefferman expansion} were
extended to pseudoconvex domains of finite type in $\mathbb{C}^{2}$
by Hsiao and the first author \cite[Theorem 2]{HsiaoSavale-2022}.
They established the full asymptotic expansion of the Bergman kernel
described along transversal paths approaching the boundary. This is
not suitable for our proof of \prettyref{thm:main theorem}. We shall
need the asymptotic expansion of the Bergman kernel, and its derivatives,
along certain critically tangent paths (see \prettyref{sec:Introdunction}
and Remark \prettyref{rem:(Critical-tangency)}) approaching the boundary.
Besides, we also need information of the leading coefficient in the
asymptotics.

To state our result, now let $D\subset\mathbb{C}^{2}$ be a smoothly bounded pseudoconvex
domain of finite type. Fix $x^{*}\in X=\partial D$ on the boundary
of the domain of type $r=r\left(x^{*}\right)$. Let $U_{1},U_{2}\coloneqq JU_{1}\in C^{\infty}\left(HX\right)$
be two local orthonormal sections of the Levi distribution and $U_{3}\in C^{\infty}\left(TX\right)$,
$U_{3}\perp HX$ to be a unit normal to the Levi distribution. One
then extends $U_{1}$ to a local unit length vector field in the interior
of $D$. Set $U_{2}=JU_{1}$ to be an extension of $U_{2}$ to the
interior of $D$. Choose an extension of $U_{3}$ of unit length and
that is orthogonal to $U_{1},U_{2}$. Set $U_{0}=-JU_{3}$ (so that
$U_{3}=JU_{0}$). It is easy to see that $U_{0}$ is of unit length
and normal to the boundary $U_{0}\perp TX$ near $x^{*}\in X$. Replacing
$U_{3}$ by $-U_{3}$ if needed, we assume $U_{0}$ is outward-pointing
to $D$. This also gives a local boundary defining function $\rho$
via $U_{0}\left(\rho\right)=1$, $\left.\rho\right|_{X}=0$. Note
that the flow of the normal vector field $U_{0}$ also gives a locally
defined projection $\pi:D\rightarrow X=\partial D$ onto the boundary.
The pairs of vector fields define CR vector fields $Z=\frac{1}{2}\left(U_{1}-iU_{2}\right),W=\frac{1}{2}\left(U_{0}-iU_{3}\right)\in T^{1,0}\mathbb{C}^{2}$. 

Next, from a system of coordinates $x=\left(x_{1},x_{2},x_{3}\right)$
on the boundary centered at $x^{*}\in X$, we assign weights to local
functions and vector fields. For that, we define the weight of $x_{1},x_{2},x_{3}$ to be $1,1,r,$ respectively; and write $w(x)=w\left(x_{1},x_{2},x_{3}\right)\coloneqq\left(1,1,r\right)$ for the augmented weight vector. Then the weight of
a monomial $x^{\alpha}$, $\alpha=(\alpha_1, \alpha_2, \alpha_3)\in\mathbb{N}_{0}^{3}$, is $w.\alpha\coloneqq\alpha_{1}+\alpha_{2}+r\alpha_{3}$. In general,
the weight $w\left(f\right)$ of a function $f\in C^{\infty}\left(X\right)$
is the minimum weight of the monomials appearing in its Taylor
series at $x^{*}=0$. Finally, the weight $w\left(U\right)$ of a
smooth vector field $U=\sum_{j=1}^{3}f_{j}\partial_{x_{j}}$ is $w\left(U\right)\coloneqq\min\left\{ w\left(f_{1}\right)-1,w\left(f_{2}\right)-1,w\left(f_{3}\right)-r\right\} $.

In \cite[Prop. 3.2]{Christ89-embedding} (see also \cite[page 23]{Baouendi-Ebenfelt-Rothschild-99})
it was shown that a coordinate system $x=\left(x_{1},x_{2},x_{3}\right)$
on the boundary centered at $x^{*}$ may be chosen so that 
\begin{align}
\left.Z\right|_{X} & =\frac{1}{2}\left[\underbrace{\partial_{x_{1}}+\left(\partial_{x_{2}}p\right)\partial_{x_{3}}-i\left(\partial_{x_{2}}-\left(\partial_{x_{1}}p\right)\partial_{x_{3}}\right)}_{\eqqcolon Z_{0}}+R\right],\label{eq:Christ normal form}
\end{align}
where $p\left(x_{1},x_{2}\right)$ is a homogeneous, subharmonic (and
non-harmonic) real polynomial of degree and weight $r$. We note that
$r$ must be even. Besides, $p$ has no purely holomorphic or anti-holomorphic
terms in $z_{1}=x_{1}+ix_{2}$ in its Taylor expansion at $0$. Moreover,
$R=\sum_{j=1}^{3}r_{j}\left(x\right)\partial_{x_{j}}$ is a real vector
field of weight $w\left(R\right)\geq0$. 

The coordinate system $\left(x_{1},x_{2},x_{3}\right)$ on the boudary
is next extended to the interior of the domain by being constant in
the normal direction $U_{0}\left(x_{j}\right)=0$, $j=1,2,3$. Then
$x'\coloneqq\left(\rho,x_{1},x_{2},x_{3}\right)$ serve as coordinates
on the interior of the domain near $x^{*}$ in which $U_{0}=\partial_{\rho}$.
We also extend the notion of weights to the new coordinate system. We define the weight of $\rho, x_{1},x_{2},x_{3}$ to be $r,1,1,r,$ respectively;
and denote by $w'(x')=w'\left(\rho,x_{1},x_{2},x_{3}\right)\coloneqq\left(r, 1,1,r\right)$ the augmented weight vector. The weight of a monomial $\rho^{\alpha_{0}}x^{\alpha}$ is then defined
as $w\left(\rho^{\alpha_{0}}x^{\alpha}\right)=w'.\alpha'\coloneqq r\alpha_{0}+\alpha_{1}+\alpha_{2}+r\alpha_{3}$. In general, letting $f$ be a smooth function near
$x^{*}$, its weight
$w(f)$ equals the minimum weight of the monomials appearing in its Taylor
series at $x^{*}$. Finally, the weight $w\left(U\right)$
of a smooth vector field $U=f_{0}\partial_{\rho}+\sum_{j=1}^{3}f_{j}\partial_{x_{j}}$
is $w\left(U\right)\coloneqq\min\left\{ w\left(f_{0}\right)-r,w\left(f_{1}\right)-1,w\left(f_{2}\right)-1,w\left(f_{3}\right)-r\right\} $.
Note that one always has $w\left(U\right)\geq-r.$ Furthermore, if $f_{0}(0)=f_{3}(0)=0$, then $w\left(U\right)>-r$.

Below $O\left(k\right)$ denotes a vector field of weight $k$ or
higher. By a rescaling of the $x_{3}$ coordinate, and at the cost
of scaling the polynomial $p\left(x_{1},x_{2}\right)$ , we may also
arrange $\left.U_{3}\right|_{x^{*}=0}=\sigma \partial_{x_{3}}$, where $\sigma$ equals $1$ or $-1$. By  the pseudoconvexity condition \prettyref{eq:Levi form}, one can indeed
show that $\sigma$ must be $1$.  To see that, we compute
\begin{align*}
\left[Z,\bar{Z}\right]= & [-\Delta p\left(z_{1}\right)\frac{i}{2}\partial_{x_{3}}]+O\left(-1\right) \\ 
=& [-\Delta p\left(z_{1}\right)\frac{i}{2} \sigma U_3 ]+O\left(-1\right)
\end{align*}
Since $i U_3=\overline{W}-W$, we obtain
\begin{align*}
-\overline{\partial}\rho\left(\left[Z,\bar{Z}\right]\right)=& \frac{\sigma}{2} \Delta p\left(z_{1}\right) \overline{W} \left(\rho\right) + O\left(r-1\right)\\
=& \frac{\sigma}{4} \Delta p\left(z_{1}\right)  + O\left(r-1\right)
\end{align*}
The pseudoconvexity condition yields $\sigma >0,$ and thus $\sigma=1.$ Therefore we
have 
\begin{equation}
U_{3}=\partial_{x_{3}}+O\left(-r+1\right).\label{eq:arrangement for U3 =00003D00003D00003D000026 c}
\end{equation}
Next let $V\in C^{\infty}\left(HX\right)$ denote another locally
defined section of the Levi distribution. This defines a local \textit{tangential
path} approaching $x^{*}$ via 
\begin{align}
z\left(\epsilon\right)\coloneqq\left(\underbrace{e^{\epsilon V}x^{*}}_{=\pi\left(z\left(\epsilon\right)\right)},\underbrace{-\epsilon^{r}}_{=\rho\left(z\left(\epsilon\right)\right)}\right) & \in D,\quad\epsilon>0.\label{eq:tangential path}
\end{align}
Note the above path is indeed tangential to the boundary; its tangent
vector at $x^{*}$ is in the Levi-distribution $\left.\frac{dz}{d\epsilon}\right|_{\epsilon=0}=V_{x^{*}}\in H_{x^{*}}X$.
The order of tangency the path makes with the boundary is the type
of the point $r\left(x^{*}\right)$. Writing $V=\sum_{j=1}^{3}g_{j}\partial_{x_{j}}$,
we associate the section $V$ with a point 
\begin{align}
z_{1,V}\coloneqq\left(x_{1,V},x_{2,V}\right)=\left(g_{1}\left(0\right),g_{2}\left(0\right)\right)\in\mathbb{R}^{2}.\label{eq:vector field V in coords}
\end{align}

In the computation of the leading asymptotics of the Bergman kernel
$K_{D}$ (see \prettyref{eq:computation of the leading term} in \prettyref{thm:HS thm}),
one will further see the appearance of the \textit{model Bergman kernel
$B_{p}$ associated to the subharmonic polynomial $p$} in \prettyref{eq:Christ normal form}.
For the readers' convenience, we briefly recall the notion of model
Bergman kernel. For that, we consider the $L^{2}$ orthogonal projector
from $L^{2}\left(\mathbb{C}_{z_{1}}\right)$ to $H_{p}^{2}$. Here
\begin{equation}\label{eqn:hp2def}
H_{p}^{2}\coloneqq\left\{ f\in L^{2}\left(\mathbb{C}_{z_{1}}\right)|~\bar{\partial}_{p}f=0\right\} ;\quad\mathrm{and}~~\bar{\partial}_{p}\coloneqq\partial_{\bar{z}_{1}}+\partial_{\bar{z}_{1}}p.
\end{equation}
Then $B_{p}$ is defined to be the Schwartz kernel of this projector.
More discussion and analysis of the model Bergman kernel follows in
\prettyref{sec: Analysis of the model}.

We now state the necessary asymptotics result for the Bergman kernel
and its derivatives. Below $\partial^{\alpha'}=\left(\frac{1}{2}U_{0}\right)^{\alpha_{0}}Z^{\alpha_{1}}\bar{Z}^{\alpha_{2}}\left(\frac{1}{2}U_{3}\right)^{\alpha_{3}}$
denotes a mixed derivative along the respective vector fields for
$\alpha'=\left(\alpha_{0},\alpha_{1},\alpha_{2},\alpha_{3}\right)$
$\in\mathbb{N}_{0}^{4}$. 
\begin{thm}
\label{thm:HS thm} Let $D\subset\mathbb{C}^{2}$ be a smoothly bounded
pseudoconvex domain of finite type. For any point $x^{*}\in X=\partial D$
on the boundary, of type $r=r\left(x^{*}\right)$, the Bergman kernel
and its derivatives satisfy the following asymptotics for each $N\in\mathbb{N}$:
\begin{align}
 & \partial^{\alpha'}K_{D}\left(z,z\right)\quad\label{eq:HS equation}\\
 & =\sum_{j=0}^{N}\frac{1}{\left(-2\rho\right){}^{2+\frac{2+w'.\alpha'}{r}-\frac{1}{r}j}}a_{j}+\sum_{j=0}^{N}a_{j}'\left(-\rho\right){}^{j}\log\left(-\rho\right)+O\left(\left(-\rho\right)^{\frac{1}{r}\left(N+1\right)-2-\frac{2+w'.\alpha'}{r}}\right),\nonumber 
\end{align}
for some set of numbers $a_{j},a_{j}'$ as $z\rightarrow x^{*}$ tangentially
to the boundary along the path $\prettyref{eq:tangential path}$.

Furthermore, the leading term can be computed in terms of the model
Bergman kernel of the subharmonic polynomial $p$ as 
\begin{equation}
a_{0}=\delta_{0\alpha_{3}}.\left[\partial_{z_{1}}^{\alpha_{1}}\partial_{\bar{z}_{1}}^{\alpha_{2}}\underbrace{\left(\frac{1}{\pi}\int_{0}^{\infty}e^{-s}s^{1+\frac{2}{r}+\alpha_{0}}B_{p}\left(s^{\frac{1}{r}}z_{1}\right)ds\right)}_{\eqqcolon\tilde{B}_{p,\alpha_{0}}\left(z_{1}\right)}\right]_{z_{1}=z_{1,V}}.\label{eq:computation of the leading term}
\end{equation}
\end{thm}

\begin{proof}
The proof is similar to \cite[Thm. 2]{HsiaoSavale-2022}. We shall
only point out the necessary modifications.

In \cite[Sec. 4 ]{HsiaoSavale-2022} the following space of symbols
$\hat{S}_{\frac{1}{r}}^{m}\left(\mathbb{C}^{2}\times\mathbb{C}^{2}\times\mathbb{R}_{t}\right)$,
$m\in\mathbb{R}$, in the variables $\left(\rho,x,\rho',y;t\right)\in\mathbb{C}_{z}^{2}\times\mathbb{C}_{w}^{2}\times\mathbb{R}_{t}$
was defined. This is the space of smooth functions satisfying the
symbolic estimates 
\begin{align}
{\scriptstyle \left|\partial_{\rho}^{\alpha_{0}}\partial_{\rho'}^{\beta_{0}}\partial_{x}^{\alpha}\partial_{y}^{\beta}\partial_{t}^{\gamma}a(\rho,x,\rho',y,t)\right|} & \leq C_{N,\alpha\beta\gamma}\left\langle t\right\rangle ^{m-\gamma+\frac{w'.\left(\alpha'+\beta'\right)}{r}}\frac{\left(1+\left|t^{\frac{1}{r}}\hat{x}\right|+\left|t^{\frac{1}{r}}\hat{y}\right|\right)^{N\left(\alpha',\beta',\gamma\right)}}{\left(1+\left|t^{\frac{1}{r}}\hat{x}-t^{\frac{1}{r}}\hat{y}\right|\right)^{-N}},\label{eq:symbolic estimates-1}
\end{align}
for each $\left(x,y,\rho,\rho',t,N\right)\in\mathbb{R}_{x,y}^{6}\times\mathbb{R}_{\rho,\rho'}^{2}\times\mathbb{R}_{t}\times\mathbb{N}$
and $\left(\alpha',\beta',\gamma\right)\in\mathbb{N}_{0}^{4}\times\mathbb{N}_{0}^{4}\times\mathbb{N}_{0}$
with $\alpha'=(\alpha_{0},\alpha),\beta'=(\beta_{0},\beta)$. Here
$N\left(\alpha',\beta',\gamma\right)\in\mathbb{N}$ depends only on
the given indices, $\left\langle t\right\rangle \coloneqq\sqrt{1+t^{2}}$
denotes the Japanese bracket while the notation $\hat{x}=\left(x_{1},x_{2}\right)$
denotes the first two coordinates of the tuple $x=\left(x_{1},x_{2},x_{3}\right)$.
Below $\hat{S}\left(\mathbb{R}_{\hat{x}}^{2}\times\mathbb{R}_{\hat{y}}^{2}\right)$
further denotes the space of restrictions of functions in $\hat{S}_{\frac{1}{r}}^{m}$
to $x_{3},y_{3},\rho,\rho'=0$ and $t=1$.

Next a generalization of this space is defined via 
\begin{equation}
\hat{S}_{\frac{1}{r}}^{m,k}\coloneqq\bigoplus_{p+q+p'+q'\leq k}\left(tx_{3}\right)^{p}\left(t\rho\right)^{q}\left(ty_{3}\right)^{p'}\left(t\rho'\right)^{q'}\hat{S}_{\frac{1}{r}}^{m},\label{eq:Smk symbol space def.}
\end{equation}
for each $\left(m,k\right)\in\mathbb{R}\times\mathbb{N}_{0}$. Finally,
the subspace of classical symbols $\hat{S}_{\frac{1}{r},{\rm cl\,}}^{m}\subset\hat{S}_{\frac{1}{r}}^{m}$
comprises of those symbols for which there exist $a_{jpp'qq'}\left(\hat{x},\hat{y}\right)\in\hat{S}\left(\mathbb{R}^{2}\times\mathbb{R}^{2}\right)$,
$j,p,p',q,q'\in\mathbb{N}_{0}$, such that the following belongs to
$\hat{S}_{\frac{1}{r}}^{m-\left(N+1\right)\frac{1}{r},N+1}$ for each
$N\in\mathbb{N}_{0}$: 
\begin{equation}
a\left(x,y,t\right)-\sum_{j=0}^{N}\sum_{p+q+p'+q'\leq j}t^{m-\frac{1}{r}j}\left(tx_{3}\right)^{p}\left(t\rho\right)^{q}\left(ty_{3}\right)^{p'}\left(t\rho'\right)^{q'}a_{jpp'qq'}\left(t^{\frac{1}{r}}\hat{x},t^{\frac{1}{r}}\hat{y}\right).\label{eq:asymbol}
\end{equation}
The space $\hat{S}_{\frac{1}{r},\textrm{cl}}^{m,k}$ is now defined
similarly to \prettyref{eq:Smk symbol space def.}. The principal
symbol of such an element $a\in\hat{S}_{\frac{1}{r},{\rm cl\,}}^{m}$
is defined to be the function 
\[
\sigma_{L}\left(a\right)\coloneqq a_{00000}\in\hat{S}\left(\mathbb{R}^{2}\times\mathbb{R}^{2}\right).
\]

Now, following the proof of \cite[Prop. 7.6]{Hsiao2010}, there exists
a smooth phase function $\Phi(z,w)$ defined locally on a neighbourhood
$U\times U$ of $\left(x^{*},x^{*}\right)$ in $\bar{D}\times\bar{D}\subset\mathbb{C}_{z}^{2}\times\mathbb{C}_{w}^{2}$
such that 
\begin{align}
 & \Phi(z,w)-x_{3}+y_{3}\label{eq:two variable phase}\\
 & =-i\rho\sqrt{-\sigma_{\triangle_{X}}(x,(0,0,1))}-i\rho'\sqrt{-\sigma_{\triangle_{X}}(y,(0,0,1))}+O(\left|\rho\right|^{2})+O(\left|\rho'\right|^{2}),\nonumber 
\end{align}
where $q_{0}(z,d_{z}\Phi)$ and $q_{0}(w,-\overline{d}_{w}\Phi)$
vanish to infinite order on $\{\rho=0\}$ and on $\{\rho'=0\}$, respectively.
Here $\triangle_{X}$ denotes the real Laplace operator on the boundary
$X=\partial D$ of the domain, while $q_{0}=\sigma\left(\Box_{f}\right)$
denotes the principal symbol of the complex Laplace-Beltrami operator
$\Box_{f}=\bar{\partial}_{f}^{*}\bar{\partial}+\bar{\partial}\bar{\partial}_{f}^{*}$
on the domain. The proofs of \cite[Lemma 17]{HsiaoSavale-2022} and
\cite[Lemma 20]{HsiaoSavale-2022} can be repeated to obtain the following
description for the Bergman kernel: for some $a\left(z,w,t\right)\in\hat{S}_{\frac{1}{r},{\rm cl\,}}^{1+\frac{2}{r}}\left(\mathbb{C}^{2}\times\mathbb{C}^{2}\times\mathbb{R}_{t}\right)$
one has 
\begin{align}
K_{D}\left(z,w\right)=\frac{1}{\pi}\int_{0}^{\infty}e^{i\Phi(z,w)t}a\left(z,w,t\right)dt\quad\left(\textrm{mod }C^{\infty}\left(\left(U\times U\right)\cap\left(\overline{D}\times\overline{D}\right)\right)\right)\label{eq:Bergman kernel description}
\end{align}
with $\sigma_{L}\left(a\right)=B_{p}$ being the model Bergman kernel
defined prior to the statement of this theorem.

We need to differentiate the last description \prettyref{eq:Bergman kernel description}.
For that, we adopt the notion of weights we defined before \prettyref{thm:HS thm}.
By construction, the chosen vector fields $\left(U_{0},Z,\overline{Z},U_{3}\right)$
have weights $\left(-r,-1,-1,-r\right)$ respectively. Furthermore,
the leading parts in their weight expansions are given by 
\begin{align}
\left(U_{0},Z,\overline{Z},U_{3}\right) & =\left(\partial_{\rho},Z_{0}+O\left(0\right),\bar{Z}_{0}+O\left(0\right),\partial_{x_{3}}+O\left(-r+1\right)\right),\label{eq:coords =00003D00003D00003D000026 vector fields along the path}
\end{align}
Here $Z_{0}\coloneqq\frac{1}{2}[\partial_{x_{1}}+\left(\partial_{x_{2}}p\right)\partial_{x_{3}}-i\left(\partial_{x_{2}}-\left(\partial_{x_{1}}p\right)\partial_{x_{3}}\right)]$
is now understood as a locally defined vector field in the interior
of the domain. Next we observe from definitions of the symbol spaces
\prettyref{eq:symbolic estimates-1}, \prettyref{eq:Smk symbol space def.}
that a vector field $U$ of weight $w\left(U\right)$ maps 
\begin{equation}
U:\hat{S}_{\frac{1}{r},{\rm cl\,}}^{m}\rightarrow\hat{S}_{\frac{1}{r},{\rm cl\,}}^{m-\frac{1}{r}w\left(U\right)}.\label{eq:weight differentiation properties of Smk}
\end{equation}
The equations \prettyref{eq:two variable phase}, \prettyref{eq:coords =00003D00003D00003D000026 vector fields along the path},
\prettyref{eq:weight differentiation properties of Smk} now allow
us to differentiate \prettyref{eq:Bergman kernel description} to
obtain: for some $a_{\alpha}\left(z;w,t\right)\in\hat{S}_{\frac{1}{r},{\rm cl\,}}^{1+\frac{2+w'.\alpha}{r},\alpha_{0}+\alpha_{3}}\left(\mathbb{C}^{2}\times\mathbb{C}^{2}\times\mathbb{R}_{t}\right)$
one has 
\begin{align}
\partial^{\alpha}K_{D}\left(z,z\right) & =\frac{1}{\pi}\int_{0}^{\infty}e^{i\Phi(z,z)t}a_{\alpha}\left(z,z,t\right)dt\quad\left(\textrm{mod }C^{\infty}\left(\left(U\times U\right)\cap\left(\overline{D}\times\overline{D}\right)\right)\right)\nonumber \\
\textrm{with }\quad a_{\alpha} & =\left(Z_{0}^{\alpha_{1}}\bar{Z}_{0}^{\alpha_{2}}B_{p}\right)t^{1+\frac{2+w'.\alpha}{r}}+\hat{S}_{\frac{1}{r},{\rm cl\,}}^{1+\frac{1+w'.\alpha}{r},\alpha_{0}+\alpha_{3}}.\label{eq:Bergman kernel der. description}
\end{align}

Recall the vector field $V=\sum_{j=1}^{3}g_{j}\partial_{x_{j}}\in C^{\infty}\left(HX\right)$
lies in the Levi distribution. By \prettyref{eq:Christ normal form},
its $\partial_{x_{3}}$-component function has weight $w\left(g_{3}\right)\geq r-1$.
Thus along the flow of $V$, and consequently along the path $z\left(\epsilon\right)$
in \prettyref{eq:tangential path}, the coordinate functions satisfy
\begin{equation}
\left(x_{1},x_{2},x_{3},\rho\right)=\left(\epsilon g_{1}\left(0\right)+O\left(\epsilon^{2}\right),\epsilon g_{2}\left(0\right)+O\left(\epsilon^{2}\right),O\left(\epsilon^{r}\right),-\epsilon^{r}\right).\label{eq:expansion of path}
\end{equation}
The last two equations \prettyref{eq:Bergman kernel der. description}
and \prettyref{eq:expansion of path} now combine to give the theorem. 
\end{proof}

Write $K_{D_{p}}\left(z_{1},z_{2}\right)$ for the Bergman kernel of
the model domain $$D_{p}\coloneqq\left\{ (z_1,z_2)| \textrm{Im}z_{2}>p\left(z_{1}\right)\right\} \subset\mathbb{C}^{2}$$
associated to the homogeneous polynomial $p$.
We note that the kernel $\tilde{B}_{p,0}$ appearing in \prettyref{eq:computation of the leading term}
above can be expressed as 
$$\tilde{B}_{p,0}\left(z_{1}\right)=K_{D_{p}}\left(z_{1},i\left[1+p\left(z_{1}\right)\right]\right).$$
This can be seen by using (\ref{eq:relation between Bergman projectors}), (\ref{eq:reln between Bergman kernels}) and \cite[Thm. 2(c)]{Haslinger98}).

\begin{rem}
\label{rem:(Critical-tangency)} (Critical tangency) The path $z\left(\epsilon\right)$
in \prettyref{eq:tangential path} is particularly chosen to be critically
tangent to the boundary. Namely its order of tangency with the boundary
is the type $r\left(x^{*}\right)$ of the boundary point $x^{*}\in\partial D$
that is being approached. This order of tangency is critical in the
sense that it is the maximum for which the expansion in \prettyref{eq:HS equation}
can be proved. As for a higher order of tangency (i.e., $\rho$ having
vanishing order higher than $r$ at $\epsilon=0$), the terms in the
symbolic expansion of $a_{\alpha}\in\hat{S}_{\frac{1}{r},{\rm cl\,}}^{1+\frac{2+w'.\alpha}{r},\alpha_{0}+\alpha_{3}}$
in \prettyref{eq:Bergman kernel der. description} become increasing
in order and not asymptotically summable. This means in particular,
the double summation in \prettyref{eq:asymbol} would be asymptotically
non-summable along the path. A critically tangent path is necessary
in our proof below since for such a path the leading coefficient \prettyref{eq:computation of the leading term}
picks up information of the model Bergman kernel at the arbitrary
tangent vector $V$. For a path tangent at a lesser order, the leading
coefficient only depends on the value of the model kernel $B_{p}$
at the origin. 
\end{rem}

\section{\label{sec: Analysis of the model} Analysis of the model kernel}

In \prettyref{sec:Bergman kernel der asymptotics}, we introduced
the model Bergman kernel $B_{p}$, corresponding to a subharmonic,
homogeneous polynomial $p\left(x_{1},x_{2}\right)$. As we see from
\prettyref{thm:HS thm}, it plays an important role in the asymptotics
of the Bergman kernel $K_{D}$ of $D$. To prepare for the proof of
\prettyref{thm:main theorem}, we need to further analyze this model
Bergman kernel $B_{p}$. For convenience, we will also write $p\left(x_{1},x_{2}\right)$
as $p\left(z_{1}\right)$, where $z_{1}=x_{1}+ix_{2}.$

\subsection{Expansion of the model kernel and first few coefficients}

First we will work out the expansion of the model Bergman kernel $B_{p}$,
and compute the values of the first few coefficients in the expansion.
As usual, for a smooth function $f$ on $\mathbb{C}_{z_{1}}$, we
write $f_{z_{1}}=\partial_{z_{1}}f=\frac{\partial f}{\partial z_{1}}$,
and likewise for $f_{\bar{z}_{1}}$ and $f_{z_{1}\bar{z}_{1}}$. 


\begin{prop}
\label{prop:model asymptotics} For any $z_{1}\in\mathbb{R}^{2}$,
with $\Delta p\left(z_{1}\right)\neq0$, the model Bergman kernel
on diagonal satisfies the asymptotics 
\begin{align}
\left[\partial_{z_{1}}^{\alpha_{1}}\partial_{\bar{z}_{1}}^{\alpha_{2}}B_{p}\right]\left(t^{\frac{1}{r}}z_{1}\right) & =\frac{t^{1-\frac{2+\left|\alpha\right|}{r}}}{2\pi}\partial_{z_{1}}^{\alpha_{1}}\partial_{\bar{z}_{1}}^{\alpha_{2}}\left[\sum_{j=0}^{N}b_{j}t^{-j}+O\left(t^{-N-1}\right)\right]\label{eq:proposition about model kernel}
\end{align}
for each $N\in\mathbb{N}$ as $t\rightarrow\infty$. Moreover, the
first four terms in the asymptotics are given by 
\begin{align*}
b_{0}=4q;\quad b_{1}=q^{-2}Q;\quad b_{2}=\frac{1}{6}\partial_{z_{1}}\partial_{\bar{z}_{1}}\left[q^{-3}Q\right];
\end{align*}
\begin{align}
b_{3}=\frac{q}{48}\left\{ [q^{-1}\partial_{z_{1}}\partial_{\bar{z}_{1}}]^{2}q^{-3}Q-q^{-4}Q\left[\partial_{z_{1}}\partial_{\bar{z}_{1}}\right]q^{-3}Q-q^{-1}[\partial_{\bar{z}_{1}}\left(q^{-3}Q\right)][\partial_{z_{1}}\left(q^{-3}Q\right)]\right\} ;\label{eq:computation of first four Bergman coefficients}
\end{align}
where $q\coloneqq\frac{1}{4}\Delta p=p_{z_{1}\bar{z}_{1}}$ and $Q\coloneqq qq_{z_{1}\bar{z}_{1}}-q_{z_{1}}q_{\bar{z}_{1}}$
are defined in terms of the polynomial $p$. 
\end{prop}

\begin{proof}
The proof uses some rescaling arguments. Following \cite[Sec. 4.1]{Marinescu-Savale18},
we introduce the rescaling operator 
$\delta_{t^{-\frac{1}{r}}}:\mathbb{C}\rightarrow\mathbb{C}$ given
by 
$\delta_{t^{-\frac{1}{r}}}\left(z_{1}\right)\coloneqq t^{-\frac{1}{r}}z_{1},t>0.$
Recall when introducing $B_{p},$ we defined $\bar{\partial}_{p}\coloneqq\partial_{\bar{z}_{1}}+\partial_{\bar{z}_{1}}p.$
The corresponding Kodaira Laplacian on functions 
$\Box_{p}=\bar{\partial}_{p}^{*}\bar{\partial}_{p}$ then gets rescaled
to the operator 
\[
\left(\delta_{t^{-\frac{1}{r}}}\right)_{*}\Box_{p}=t^{-\frac{2}{r}}\Box_{t}
\]
where $\Box_{t}\coloneqq\bar{\partial}_{t}^{*}\bar{\partial}_{t}$,
and $\bar{\partial}_{t}\coloneqq\partial_{\bar{z}_{1}}+t\left(\partial_{\bar{z}_{1}}p\right)$.

Recall right before Theorem \ref{thm:HS thm}, we defined $B_p$ as the model Bergman kernel
associated to the subharmonic polynomial $p$. We now pause to introduce two more Bergman type kernel functions that
are defined similarly as $B_{p}.$ 

Firstly denote by the shorthand
$\mathcal{B}_{t}\coloneqq B_{tp}$, $t>0$, the model Bergman kernel
associated to the rescaled homogeneous polynomial $tp.$ Next define
the weighted space $L_{tp}^{2}\left(\mathbb{C}_{z_{1}}\right)\coloneqq\left\{ f|e^{-tp}f\in L^{2}\left(\mathbb{C}_{z_{1}}\right)\right\} $,
and denote by $\mathcal{O}\left(\mathbb{C}_{z_{1}}\right)$ the space
of entire functions on $\mathbb{C}_{z_{1}}.$ The $L^{2}$ orthogonal
projector $\mathcal{B}_{t}^{p}$ from $L_{tp}^{2}\left(\mathbb{C}_{z_{1}}\right)$
to $L_{tp}^{2}\left(\mathbb{C}_{z_{1}}\right)\cap\mathcal{O}\left(\mathbb{C}_{z_{1}}\right)$
is then seen to be related to the kernel $B_{t}$ by the relation
\begin{equation}
	\mathcal{B}_{t}\left(z_{1},z_{1}'\right)=e^{-tp\left(z_{1}\right)-tp\left(z_{1}'\right)}\mathcal{B}_{t}^{p}\left(z_{1},z_{1}'\right).\label{eq:relation between Bergman projectors}
\end{equation}
This follows from the simple observation that multiplication by $e^{-tp}$ is
a isomorphism from $L_{tp}^{2}\left(\mathbb{C}_{z_{1}}\right)$ to
$L^{2}\left(\mathbb{C}_{z_{1}}\right)$, and from $L_{tp}^{2}\left(\mathbb{C}_{z_{1}}\right)\cap\mathcal{O}\left(\mathbb{C}_{z_{1}}\right)$
to $H_{tp}^{2}$, respectively. Here recall $H_{tp}^{2}$ is as defined in (\ref{eqn:hp2def}) with $p$ replaced by $tp$.

Moreover, $\mathcal{B}_{t}$ can be equivalently understood as the
Bergman projector for the trivial holomorphic line bundle on $\mathbb{C}$
with Hermitian metric $h_{t}=e^{-tp}$. The curvature of this metric
is $t\underbrace{\left(2\partial_{z_{1}}\partial_{\bar{z}_{1}}p\right)}_{=\frac{1}{2}\Delta p}dz_{1}\wedge d\bar{z}_{1}$.
Its eigenvalue is $\Delta p$. In \cite[Thm. 14]{HsiaoSavale-2022},
$\mathcal{B}_{t}$  is related to $\mathcal{B}_{1}=B_p$ via 
\begin{equation}
	B_{p}\left(t^{\frac{1}{r}}z_{1},t^{\frac{1}{r}}z_{1}'\right)=t^{-\frac{2}{r}}\mathcal{B}_{t}\left(z_{1},z_{1}'\right).\label{eq:reln between Bergman kernels}
\end{equation}
Furthermore, in its proof the following spectral gap property for
$\Box_{t}$ 
was observed 
\[
\textrm{Spec}\left(\Box_{t}\right)\subset\left\{ 0\right\} \cup\left[c_{1}t^{2/r}-c_{2},\infty\right)
\]
for some $c_{1},c_{2}>0$.

At a point $z_{1}\in\mathbb{C}$, where $\Delta p\left(z_{1}\right)\neq0$,
the asymptotics of $\mathcal{B}_{t}\left(z_{1},z_{1}\right)$ as $t\rightarrow\infty$
are thus the standard asymptotics for the Bergman kernel on tensor
powers of a positive line bundle (cf. \cite[Thm. 1.6]{Hsiao-Marinescu-2014}).
There is an asymptotic expansion 
\begin{equation}
	\partial_{z_{1}}^{\alpha_{1}}\partial_{\bar{z}_{1}}^{\alpha_{2}}\mathcal{B}_{t}\left(z_{1}\right)=\frac{t}{2\pi}\partial_{z_{1}}^{\alpha_{1}}\partial_{\bar{z}_{1}}^{\alpha_{2}}\left[\sum_{j=0}^{N}b_{j}t^{-j}+O\left(t^{-N-1}\right)\right]\label{eq:standard derivative asymptotics}
\end{equation}
for each $N\in\mathbb{N}$ as $t\rightarrow\infty$. The last two
equations \prettyref{eq:reln between Bergman kernels} and \prettyref{eq:standard derivative asymptotics}
combine to prove \prettyref{eq:proposition about model kernel}.

It remains to compute the first four coefficients in \prettyref{eq:standard derivative asymptotics}.
For that we will make use of \prettyref{eq:relation between Bergman projectors},
by which it suffices to find the corresponding coefficients in the
expansion of $\mathcal{B}_{t}^{p}.$ The computations for the latter
can be found in \cite[(6.2) and Theorem 9]{Englis2000}. In order
to see the specialization of the formulas therein to the special case
here, 
we note the Kähler metric $g=\partial\bar{\partial}p$ with potential
$p$ has component $g_{1\bar{1}}=q=\partial_{z_{1}}\partial_{\bar{z}_{1}}p$.
The only non-zero Christoffel symbols are 
$\overline{\Gamma_{11}^{1}}=\Gamma_{\bar{1}\bar{1}}^{\bar{1}}=q^{-1}\partial_{\bar{z}_{1}}q.$
Furthermore, the only non-zero components of the Riemannian, Ricci
and scalar curvatures respectively are given by the following. Here
we follow the convention of curvatures in \cite[pp. 6]{Englis2000},
which may differ from that of some other papers by a negative sign.
\begin{align*}
	R_{1\bar{1}1\bar{1}}=\partial_{z_{1}}\partial_{\bar{z}_{1}}q-q^{-1}\left(\partial_{z_{1}}q\right)\left(\partial_{\bar{z}_{1}}q\right)=q^{-1}Q;\quad\textrm{Ric}_{1\bar{1}}=q^{-2}Q;\quad R=q^{-3}Q.
\end{align*}
The corresponding Laplace operator $L_{1}$ of \cite[(2.10)]{Englis2000}
in our special context is given by $L_{1}=q^{-1}\partial_{z_{1}}\partial_{\bar{z}_{1}}$.
We now bring these specializations into \cite[(6.2) and Theorem 9]{Englis2000}
to obtain the values of the coefficients $b_{0}$, $b_{1}$, $b_{2}$
and $b_{3}$. For instance,  we note that the tensors appearing the
computation of $b_{3}$ are the following ones, in the notation of \cite[Theorem 9]{Englis2000}: $\sigma_1=\cdots=\sigma_7=\sigma_{15}=q^{-9}Q^3, \sigma_{8}=\sigma_{9}=\sigma_{10}=q^{-4}Q\left[\partial_{z_{1}}\partial_{\bar{z}_{1}}\right]q^{-3}Q$,
$\sigma_{12}=\sigma_{13}=q^{-1}\left[\partial_{\bar{z}_{1}}\left(q^{-3}Q\right)\right]\left[\partial_{z_{1}}\left(q^{-3}Q\right)\right]$
and $\sigma_{14}=\left[q^{-1}\partial_{z_{1}}\partial_{\bar{z}_{1}}\right]^{2}q^{-3}Q$.
\end{proof}
\begin{rem}
Although we computed the values of $b_{0},\cdots,b_{3}$ in Proposition
\prettyref{prop:model asymptotics}, we will only use $b_{3}$ in
the proof of \prettyref{thm:main theorem}. 
\end{rem}

\subsection{Models with vanishing expansion coefficients}

Having shown that the model kernel $B_{p}\left(t^{\frac{1}{r}}z_{1}\right)$
admits an asymptotic expansion at $t\rightarrow\infty$, we ask when
the terms of the asymptotic expansion are eventually zero, or in other
words, $b_{j}=0$ for $j$ sufficiently large. This is relevant to
our theorem below. We prove the following somewhat surprising result
which shows the vanishing of the third coefficient is already restrictive.
As above, let $p\left(x_{1},x_{2}\right)$ be a subharmonic and non-harmonic
homogeneous polynomial of degree $r$. 
\begin{thm}
\label{thm: thm about vanish coeff.} Suppose the third term $b_{3}$
vanishes in the asymptotic expansion \prettyref{eq:proposition about model kernel}
of the model kernel $B_{p}$ corresponding to $p$. Then there exists
some real number $c_{0}>0$ such that $q=c_{0}\left(z_{1}\bar{z}_{1}\right)^{\frac{r}{2}-1}.$
Here as before, $q\coloneqq\frac{1}{4}\Delta p$. 
\end{thm}

To prove the theorem, we carry out some Hermitian analysis. For that,
we start with a few definitions and lemmas. In the remainder of this
subsection, we will write $z$ instead of $z_{1}$ for simplicity. 
\begin{defn}
Let $f\in\mathbb{C}[z,\zeta]$ be a polynomial of two variables. Fix
$a\in\mathbb{C}.$ Let $k\in\mathbb{N}_{0}$ and $\lambda\in\mathbb{C}.$
We say $f$ is \textit{divisible} by $(z+a\zeta)^{k}$ with coefficient
$\lambda,$ denoted by $f\sim D_{a}(k,\lambda),$ if $f(z,\zeta)=(z+a\zeta)^{k}\hat{f}(z,\zeta)$
for some $\hat{f}\in\mathbb{C}[z,\zeta]$ with $\hat{f}(-a,1)=\lambda.$ 
\end{defn}

It is clear that if $f\sim D_{a}(k,\lambda)$ with $k\geq1$, then
we have $f\sim D_{a}(k-1,0).$ In the following, we say $f\in\mathbb{C}[z,\zeta]$
is \textit{Hermitian} if $f(z,\bar{z})$ is real-valued for every
$z\in\mathbb{C}.$ 
\begin{lem}
\label{lm3} Let $f\in\mathbb{C}[z,\zeta]$ be a nonconstant Hermitian
homogeneous polynomial of two variables. Then there exist $a\in\mathbb{C},k\geq1$
and a nonzero $\lambda\in\mathbb{C}$ such that $f\sim D_{a}(k,\lambda).$
Moreover, if $f\neq cz^{m}\zeta^{m}$ for every real number $c\neq0$
and integer $m\geq1$, then we can further choose $a\neq0.$ 
\end{lem}

\begin{proof}
Write $d$ for the degree of $f$. Since $f$ is homogeneous, we have
\begin{equation}
f(z,\zeta)=\zeta^{d}f\left(\frac{z}{\zeta},1\right).\label{eqnq1-1}
\end{equation}
By assumption, $f(\eta,1)\in\mathbb{C}[\eta]$ is nonconstant, for
otherwise $f(z,\zeta)$ is not Hermitian. Using the fundamental theorem
of algebra, we may write 
\begin{equation}
f(\eta,1)=c\eta^{m}\prod_{j=1}^{l}(\eta-a_{j})^{k_{j}}.\label{eqnq2-1}
\end{equation}
Here $c\in\mathbb{C}$ is nonzero, and $m,l\geq0$ and $k_{j}\geq1$
satisfy $m+\sum_{j=1}^{l}k_{j}\leq d.$ Moreover, $a_{j}'$s are distinct
nonzero complex numbers. When $l=0,$ the above equation is understood
as $f(\eta,1)=c\eta^{m}.$ By (\ref{eqnq1-1}) and (\ref{eqnq2-1}),
we have 
\begin{equation}
f(z,\zeta)=cz^{m}\zeta^{n}\prod_{j=1}^{l}(z-a_{j}\zeta)^{k_{j}},\quad\text{where}\quad n=d-m-\sum_{j=1}^{l}k_{j}.\label{eqnqzx-1}
\end{equation}

We first consider the case where $l=0.$ In this case, $f(z,\zeta)=cz^{m}\zeta^{n}.$
Since $f$ is nonconstant and Hermitian, we must have $c\in\mathbb{R},c\neq0,$
and $n=m\geq1$. The conclusion of the lemma follows if we choose
$a=0,k=m\geq1,\lambda=c\neq0.$

We next assume $l\geq1.$ Then by (\ref{eqnqzx-1}), the conclusion
of the lemma follows if we choose $a=-a_{1}\neq0,k=k_{1}\geq1,\lambda=ca_{1}^{m}\prod_{j=2}^{l}(a_{1}-a_{j})^{k_{j}}\neq0$.
This proves the first part of Lemma \ref{lm3}.

Note if $f$ is not a multiple of $z^{m}\zeta^{m}$ for any integer
$m$, then it can only be the latter case, and this establishes the
second part of Lemma \ref{lm3}. 
\end{proof}
We next extend the above definition to rational functions. 
\begin{defn}
Let $g\in\mathbb{C}(z,\zeta)$ be a rational function. Write $g=\frac{f_{1}}{f_{2}},$
where $f_{1},f_{2}\in\mathbb{C}[z,\zeta]$ and $f_{2}\neq0.$ If $f_{i}\sim D_{a}(k_{i},\lambda_{i}),1\leq i\leq2,$
with $k_{1},k_{2}\geq0$ and $\lambda_{2}\neq0,$ then we say $g\sim D_{a}(k_{1}-k_{2},\frac{\lambda_{1}}{\lambda_{2}}).$
Note that $k_{1}-k_{2}$ could be negative. 
\end{defn}

Note if $g\in\mathbb{C}(z,\zeta)$ and $g\sim D_{a}(k,\lambda)$,
then we have $g\sim D_{a}(k-1,0).$ We next make a few more observations.

\begin{lem}
\label{lm5} If $g\in\mathbb{C}(z,\zeta)$ and $g\sim D_{a}(k,\lambda)$
for some $a\in\mathbb{C}$, then the following hold:

(1) $\partial_{z}g\sim D_{a}(k-1,k\lambda)$ and $\partial_{\zeta}g\sim D_{a}(k-1,ak\lambda);$

(2) $\partial_{z}\partial_{\zeta}g\sim D_{a}(k-2,ak(k-1)\lambda).$ 
\end{lem}

\begin{proof}
Write $g=\frac{f_{1}}{f_{2}}$ with $f_{1},f_{2}\in\mathbb{C}[z,\zeta],f_{2}\neq0.$
Write $f_{i}=(z+a\zeta)^{k_{i}}h_{i}$ for $1\leq i\leq2,$ where
$h_{1},h_{2}\in\mathbb{C}[z,\zeta],k_{1},k_{2}\geq0,k_{1}-k_{2}=k$
and $h_{2}(-a,1)\neq0,\frac{h_{1}(-a,1)}{h_{2}(-a,1)}=\lambda.$ We compute
\begin{align*}
\partial_{z}g & =\frac{f_{2}\partial_{z}f_{1}-f_{1}\partial_{z}f_{2}}{f_{2}^{2}}\\
 & =\frac{(k_{1}-k_{2})(z+a\zeta)^{k_{1}+k_{2}-1}h_{1}h_{2}+(z+a\zeta)^{k_{1}+k_{2}}(h_{2}\partial_{z}h_{1}-h_{1}\partial_{z}h_{2})}{(z+a\zeta)^{2k_{2}}h_{2}^{2}}.
\end{align*}
Then it is clear that $\partial_{z}g\sim D_{a}(k-1,k\lambda)$. Similarly
one can show $\partial_{\zeta}g\sim D_{a}(k-1,ak\lambda).$ This finishes
the proof of part (1). The conclusion in part (2) follows immediately
from part (1). 
\end{proof}
The statements in the next lemma follow from direct computations.
We omit the proof. 
\begin{lem}
\label{lm6} Let $g_{1},g_{2}\in\mathbb{C}(z,\zeta)$ and $a\in\mathbb{C}$.
Assume $g_{i}\sim D_{a}(k_{i},\lambda_{i})$ for $1\leq i\leq2$ where
$k_{i}\in\mathbb{Z}$ and $\lambda_{i}\in\mathbb{C}$, then the following
hold:

(1) $g_{1}g_{2}\sim D_{a}(k_{1}+k_{2},\lambda_{1}\lambda_{2});$

(2) $cg_{1}\sim D_{a}(k_{1},c\lambda_{1})$ for any complex number
$c;$

(3) $g_{1}+g_{2}\sim D_{a}(k_{1},\lambda_{1}+\lambda_{2})$ if $k_{1}=k_{2};$
and $g_{1}+g_{2}\sim D_{a}(k_{1},\lambda_{1})$ if $k_{1}<k_{2};$

(4) In addition assume $\lambda_{2}\neq0.$ Then $\frac{g_{1}}{g_{2}}\sim D_{a}\left(k_{1}-k_{2},\frac{\lambda_{1}}{\lambda_{2}}\right).$ 
\end{lem}

We are now ready to prove \prettyref{thm: thm about vanish coeff.}. 
\begin{proof}[Proof of \prettyref{thm: thm about vanish coeff.}]
Recall $q=\partial_{z}\partial_{\bar{z}}p$ and $Q=q(\partial_{z}\partial_{\bar{z}}p)-(\partial_{z}q)(\partial_{\bar{z}}q)$
are real polynomials in $\mathbb{C}[z,\bar{z}]$. Note we can assume
$q$ is nonconstant, for otherwise the conclusion is trivial. 
We will identify $p(z,\bar{z})\in\mathbb{C}[z,\bar{z}]$ with its
complexification $p(z,\zeta)\in\mathbb{C}[z,\zeta]$ (where we replace
$\bar{z}$ by a new variable $\zeta$). Moreover, since $p(z,\bar{z})$
is real-valued, $p(z,\zeta)$ is Hermitian. Likewise for $q(z,\bar{z})$
and $Q(z,\bar{z}).$ To establish \prettyref{thm: thm about vanish coeff.},
it suffices to show that $q(z,\zeta)=c_{0}z^{m}\zeta^{m}$ for some
constant $c_{0}$ and integer $m\geq1$. Seeking a contraction, suppose
the conclusion fails. Then by Lemma \ref{lm3}, we can find some complex
numbers $a\neq0,\lambda\neq0$, and some integer $k\geq1$ such that
$q\sim D_{a}(k,\lambda).$ That is, we can write $q(z,\zeta)=(z+a\zeta)^{k}h,$
where $h\in\mathbb{C}[z,\zeta]$ and $h(-a,1)=\lambda.$ A direct
computation yields the following holds for some $\hat{h}\in\mathbb{C}[z,\zeta].$
\[
Q(z,\zeta)=-ak(z+a\zeta)^{2k-2}h^{2}+(z+a\zeta)^{2k-1}\hat{h}.
\]
Thus we have 
$Q\sim D_{a}\left(2k-2,-ak\lambda^{2}\right).$ 
By assumption $b_{3}\equiv0.$ We multiply it by $\frac{48}{q}$ and
use the standard complexification to get 
\begin{equation}
\left[q^{-1}\partial_{z}\partial_{\zeta}\right]^{2}q^{-3}Q-q^{-4}Q\left[\partial_{z}\partial_{\zeta}\right]q^{-3}Q-q^{-1}\left[\partial_{\zeta}\left(q^{-3}Q\right)\right]\left[\partial_{z}\left(q^{-3}Q\right)\right]=0.\label{eqncomb3}
\end{equation}
On the other hand, by Lemma \ref{lm6}, 
$q^{3}\sim D_{a}\left(3k,\lambda^{3}\right)$ and $q^{-3}Q\sim D_{a}\left(-k-2,-\frac{ak}{\lambda}\right).$
Then by Lemma \ref{lm5}, 
\[
\partial_{z}\left(q^{-3}Q\right)\sim D_{a}\left(-k-3,\frac{ak(k+2)}{\lambda}\right);\quad\partial_{\zeta}\left(q^{-3}Q\right)\sim D_{a}\left(-k-3,\frac{a^{2}k(k+2)}{\lambda}\right).
\]
Using the above and Lemma \ref{lm6}, we can compute the last term
on the left hand side of (\ref{eqncomb3}): 
\[
-q^{-1}\left[\partial_{\zeta}\left(q^{-3}Q\right)\right]\left[\partial_{z}\left(q^{-3}Q\right)\right]\sim D_{a}\left(-3k-6,-\frac{a^{3}k^{2}(k+2)^{2}}{\lambda^{3}}\right).
\]
Similarly, we compute the first two terms on the left hand side of
(\ref{eqncomb3}): 
\[
\left[q^{-1}\partial_{z}\partial_{\zeta}\right]^{2}q^{-3}Q\sim D_{a}\left(-3k-6,-\frac{a^{3}k(k+2)(k+3)(2k+4)(2k+5)}{\lambda^{3}}\right);
\]
\[
-q^{-4}Q\left[\partial_{z}\partial_{\zeta}\right]q^{-3}Q\sim D_{a}\left(-3k-6,-\frac{a^{3}k^{2}(k+2)(k+3)}{\lambda^{3}}\right).
\]
Consequently, the left hand side of (\ref{eqncomb3}) equals to $D_{a}(-3k-6,T),$
where 
\[
T=-\frac{a^{3}k(k+2)}{\lambda^{3}}\left[k(k+2)+(k+3)(2k+4)(2k+5)+k(k+3)\right]\neq0.
\]
This means the left hand side of (\ref{eqncomb3}) is nonzero, a contradiction.
The proof is completed. 
\end{proof}

\begin{rem*}
It would be interesting to compare our Proposition \ref{prop:model asymptotics} and Theorem \ref{thm: thm about vanish coeff.} with the classical work of Bedford and Pinchuk \cite{BedfordPinchuk89}, where the same model domain (corresponding to the $p$ in Theorem \ref{thm: thm about vanish coeff.}) appears. Both proofs use some rescaling arguments, however they are different in nature. The approach in \cite{BedfordPinchuk89} 
utilizes the noncompact automorphism group to scale the domain to the model domain. On the other hand, in our Proposition \ref{prop:model asymptotics}, we use the rescaling/dilation operator $\delta_{t^{-\frac{1}{r}}}$ to get asymptotic information on the model Bergman kernel.
\end{rem*}

\subsection{The case $p=\frac{c}{2}\left(z_{1}\bar{z}_{1}\right)^{\frac{r}{2}}$}

We next consider the particular case when $p=\frac{c}{2}\left(z_{1}\bar{z}_{1}\right)^{\frac{r}{2}}$
for $c>0$ (recall $r$ must be even). Here it becomes possible to
compute the Bergman kernel $B_{p}$ explicitly. 
\begin{thm}
\label{thm:Monomial model calc.} The model Bergman kernel corresponding
to the homogeneous subhamonic polynomial $p=\frac{c}{2}\left(z_{1}\bar{z}_{1}\right)^{\frac{r}{2}}$
is given by 
\begin{align}
B_{p}\left(z_{1},z'_{1}\right) & =\frac{re^{-\left[p\left(z_{1}\right)+p\left(z_{1}'\right)\right]}c^{\frac{2}{r}}}{2\pi}G\left(c^{\frac{2}{r}}z_{1}\overline{z_{1}'}\right),\quad\textrm{where }\label{eq:model Bergman kernel-1}\\
G\left(x\right) & \coloneqq\sum_{\alpha=0}^{\frac{r}{2}-1}\frac{x^{\alpha}}{\Gamma\left(\frac{2\left(\alpha+1\right)}{r}\right)}+x^{\frac{r}{2}-1}e^{x^{\frac{r}{2}}}\left[\sum_{\alpha=0}^{\frac{r}{2}-1}\frac{\Gamma\left(\frac{2\left(\alpha+1\right)}{r}\right)-\Gamma\left(\frac{2\left(\alpha+1\right)}{r},x^{\frac{r}{2}}\right)}{\Gamma\left(\frac{2\left(\alpha+1\right)}{r}\right)}\right]\label{eq:function G}
\end{align}
is given in terms of the incomplete gamma function $\Gamma\left(a,u\right)\coloneqq\int_{u}^{\infty}t^{a-1}e^{-t}dt$,
$u>0$. 
\end{thm}

\begin{proof}
From the formulas $\Box_{p}=\bar{\partial}_{p}^{*}\bar{\partial}_{p}$
and $\bar{\partial}_{p}\coloneqq\partial_{\bar{z}_{1}}+\partial_{\bar{z}_{1}}p=\partial_{\bar{z}_{1}}+\frac{cr}{4}z_{1}^{\frac{r}{2}}\bar{z}_{1}^{\frac{r}{2}-1}$,
an orthonormal basis for $\textrm{ker}\left(\Box_{p}\right)$ is easily
found to be 
\begin{align*}
s_{\alpha} & \coloneqq\left(\frac{1}{2\pi}\frac{r}{\Gamma\left(\frac{2\left(\alpha+1\right)}{r}\right)}c^{\frac{2\left(\alpha+1\right)}{r}}\right)^{1/2}z_{1}^{\alpha}e^{-p},\quad\alpha\in\mathbb{N}_{0}.
\end{align*}
Since $B_{p}=\sum s_{\alpha}\overline{s_{\alpha}}$, we have 
\begin{align}
B_{p}\left(z_{1},z_{1}'\right) & =\frac{re^{-\left[p\left(z_{1}\right)+p\left(z'_{1}\right)\right]}}{2\pi}\sum_{\alpha\in\mathbb{N}_{0}}\frac{1}{\Gamma\left(\frac{2\left(\alpha+1\right)}{r}\right)}c^{\frac{2\left(\alpha+1\right)}{r}}\left(z_{1}\overline{z_{1}'}\right)^{\alpha}.\label{eq:Bergman kernel orthogonal basis}
\end{align}
To compute the above in a closed form, consider the series 
\begin{align*}
F\left(y\right)\coloneqq\sum_{\alpha=0}^{\infty}\frac{y^{\frac{\alpha+1}{s}-1}}{\Gamma\left(\frac{\alpha+1}{s}\right)}=\sum_{\alpha=0}^{s-1}\frac{y^{\frac{\alpha+1}{s}-1}}{\Gamma\left(\frac{\alpha+1}{s}\right)}+\underbrace{\sum_{\alpha=s}^{\infty}\frac{y^{\frac{\alpha+1}{s}-1}}{\Gamma\left(\frac{\alpha+1}{s}\right)}}_{F_{0}\left(y\right)\coloneqq},
\end{align*}
for $s=\frac{r}{2}$. Differentiating the second term in the series
and using $\Gamma\left(z+1\right)=z\Gamma\left(z\right)$ yields $F_{0}'\left(y\right)=F_{0}\left(y\right)+\sum_{\alpha=0}^{s-1}\frac{y^{\frac{\alpha+1}{s}-1}}{\Gamma\left(\frac{\alpha+1}{s}\right)}$
for $y>0.$ This ODE can be solved (uniquely) with the boundary condition
$F_{0}\left(0\right)=0$ to give 
\begin{align}
F_{0}\left(y\right)=e^{y}\left[\sum_{\alpha=0}^{s-1}\frac{\Gamma\left(\frac{\alpha+1}{s}\right)-\Gamma\left(\frac{\alpha+1}{s},y\right)}{\Gamma\left(\frac{\alpha+1}{s}\right)}\right]\label{eq:important series}
\end{align}
in terms of the incomplete gamma function. Thus in particular we have
computed $F\left(y\right)\coloneqq y^{\frac{1}{s}-1}G\left(y^{\frac{1}{s}}\right)$,
where $G$ is as defined in \prettyref{eq:function G}. Finally we
note from \prettyref{eq:Bergman kernel orthogonal basis} that 
\[
B_{p}\left(z,z'\right)=\frac{re^{-\left[p\left(z_{1}\right)+p\left(z'_{1}\right)\right]}c^{\frac{2}{r}}}{2\pi}x^{s-1}F\left(x^{s}\right),
\]
for $x=c^{\frac{2}{r}}z_{1}\overline{z'_{1}}$, completing the proof. 
\end{proof}

\section{\label{sec:Main theorem proof} Proof of the main theorem}

In this section we finally prove \prettyref{thm:main theorem}. 
\begin{proof}[Proof of \prettyref{thm:main theorem}]
It suffices to show that $D$ is strongly pseudoconvex, or the type
$r=2$ along the boundary, as thereafter one can apply Fu-Wong \cite{Fu-Wong97}
and Nemirovski-Shafikov \cite{Nemirovski-Shafikov-2006}. To this
end, suppose $x^{*}\in\partial D$ is a point on the boundary of type
$r=r\left(x^{*}\right)\geq2$. By \prettyref{eq:Monge Ampere determinant}
and Proposition \prettyref{cor:Cor KE iff B fn const}, under the
assumption of \prettyref{thm:main theorem}, the Bergman kernel $K=K_{D}$
of the domain satisfies the following Monge-Ampère equation inside
$D.$ 
\begin{align}
J\left(K\right) & =\det{\begin{pmatrix}K & \bar{Z}K & \bar{W}K\\
ZK & \left(Z\bar{Z}-\left[Z,\bar{Z}\right]^{0,1}\right)K & \left(Z\bar{W}-\left[Z,\bar{W}\right]^{0,1}\right)K\\
WK & \left(W\bar{Z}-\left[W,\bar{Z}\right]^{0,1}\right)K & \left(W\bar{W}-\left[W,\bar{W}\right]^{0,1}\right)K
\end{pmatrix}}\label{eq:MA equation for K}\\
 & =\frac{9\pi^{2}}{2}K^{4}.\nonumber 
\end{align}
Here we have used the orthonormal frame of $T^{1,0}\mathbb{C}^{2}$
given by $Z=\frac{1}{2}\left(U_{1}-iU_{2}\right)$, $W=\frac{1}{2}\left(U_{0}-iU_{3}\right)$
defined prior to \prettyref{thm:HS thm}. Using \prettyref{eq:Christ normal form}
and \prettyref{eq:coords =00003D00003D00003D000026 vector fields along the path},
we compute the $\left(0,1\right)$ components of the commutators above:
\begin{align*}
\left[Z,\bar{Z}\right]^{0,1}= & \left[-\Delta p\left(z_{1}\right)\frac{i}{2}\partial_{x_{3}}\right]^{0,1}+O\left(-1\right)\\
= & \frac{\Delta p\left(z_{1}\right)}{2}\left(W-\bar{W}\right)^{0,1}+O\left(-1\right)\\
= & -\frac{\Delta p\left(z_{1}\right)}{2}\bar{W}+O\left(-1\right);\\
\left[Z,\bar{W}\right]^{0,1}= & ~O\left(-r\right);\quad\left[W,\bar{Z}\right]^{0,1}=O\left(-r\right);\quad\left[W,\bar{W}\right]^{0,1}=O\left(-2r+1\right).
\end{align*}

This allows us to compute the most singular term in the asymptotics
of both sides of \prettyref{eq:MA equation for K} as $z\rightarrow x^{*}$
along the tangential path $z\left(\epsilon\right)$ in \prettyref{eq:tangential path}.
By \prettyref{thm:HS thm}, one obtains along $z\left(\epsilon\right)$,
\begin{align*}
{\scriptstyle J\left(K\right)=\left[\left(-2\rho\right)^{-2-\frac{2}{r}}\right]^{4}\left[\det\begin{pmatrix}\tilde{B}_{p,0} & \partial_{\bar{z}_{1}}\tilde{B}_{p,0} & \tilde{B}_{p,1}\\
\partial_{z_{1}}\tilde{B}_{p,0} & \partial_{z_{1}}\partial_{\bar{z}_{1}}\tilde{B}_{p,0}+\left[\frac{\Delta p}{2}\right]\tilde{B}_{p,1} & \partial_{z_{1}}\tilde{B}_{p,1}\\
\tilde{B}_{p,1} & \partial_{\bar{z}_{1}}\tilde{B}_{p,1} & \tilde{B}_{p,2}
\end{pmatrix}\left(z_{1,V}\right)+o_{\epsilon}\left(1\right)\right]}
\end{align*}

\begin{align*}
K^{4}=\left[\left(-2\rho\right)^{-2-\frac{2}{r}}\right]^{4}\left[\tilde{B}_{p,0}\left(z_{1,V}\right)^{4}+o_{\epsilon}\left(1\right)\right].
\end{align*}
Here we say a function $\phi$ is $o_{\epsilon}\left(1\right)$ if
$\phi(\epsilon)$ goes to $0$ as $\epsilon\rightarrow0^{+}.$ (Recall
$\rho=-\epsilon^{r}$ along the path). Thus comparing the leading
coefficients in the asymptotics gives the following equation 
\begin{equation}
\det\begin{pmatrix}\tilde{B}_{p,0} & \partial_{\bar{z}_{1}}\tilde{B}_{p,0} & \tilde{B}_{p,1}\\
\partial_{z_{1}}\tilde{B}_{p,0} & \partial_{z_{1}}\partial_{\bar{z}_{1}}\tilde{B}_{p,0}+\left[\frac{\Delta p}{2}\right]\tilde{B}_{p,1} & \partial_{z_{1}}\tilde{B}_{p,1}\\
\tilde{B}_{p,1} & \partial_{\bar{z}_{1}}\tilde{B}_{p,1} & \tilde{B}_{p,2}
\end{pmatrix}\left(z_{1}\right)=\frac{9\pi^{2}}{2}\tilde{B}_{p,0}\left(z_{1}\right)^{4},\label{eq:comparing coeff. eqn.}
\end{equation}
at each $z_{1}\in\mathbb{R}^{2}$, for the model Bergman kernel. Here
$\tilde{B}_{p,\alpha_{0}}$ is as defined in \prettyref{eq:computation of the leading term}.

Finally, one chooses $z_{1}$ such that $\Delta p\left(z_{1}\right)\neq0$
and substitutes $z_{1}\mapsto t^{\frac{1}{r}}z_{1}$ in the last equation
\prettyref{eq:comparing coeff. eqn.} above for the model. The terms
involved in the above equation are then of the following form from
the definition \prettyref{eq:computation of the leading term}. 
\begin{align*}
\tilde{B}_{p,\alpha_{0}}\left(t^{\frac{1}{r}}z_{1}\right) & =\frac{1}{\pi}\int_{0}^{\infty}e^{-s}s^{1+\frac{2}{r}+\alpha_{0}}B_{p}\left(s^{\frac{1}{r}}t^{\frac{1}{r}}z_{1}\right)ds\\
 & =\frac{t^{-2-\frac{2}{r}-\alpha_{0}}}{\pi}\int_{0}^{\infty}e^{-\frac{\tau}{t}}\tau^{1+\frac{2}{r}+\alpha_{0}}B_{p}\left(\tau^{\frac{1}{r}}z_{1}\right)d\tau
\end{align*}

Next we use Proposition \prettyref{prop:model asymptotics} to obtain
an asymptotic expansion for the above and its derivatives. Namely,
since the kernel $\tau^{1+\frac{2}{r}+\alpha_{0}}B_{p}\left(\tau^{\frac{1}{r}}z_{1}\right)$
 is a classical symbol in $S_{\tau,\textrm{cl}}^{2+\alpha_{0}}$  by Proposition
\prettyref{prop:model asymptotics}, the standard asymptotics
for its Laplace transform (e.g., using \cite[eqn. 1.6]{Boutet-Sjostrand76})
yield
\begin{align}
	& \left[\partial_{z_{1}}^{\alpha_{1}}\partial_{\bar{z}_{1}}^{\alpha_{2}}\tilde{B}_{p,\alpha_{0}}\right]\left(t^{\frac{1}{r}}z_{1}\right)\label{eq:asymptotic of det terms}\\
	& =t^{1-\frac{2+\alpha_{1}+\alpha_{2}}{r}}\left[\sum_{j=0}^{N+2+\alpha_{0}}c_{j}t^{-j}+\sum_{j=0}^{N}d_{j}t^{-\left(3+\alpha_{0}+j\right)}\ln t+O\left(t^{-\left(3+\alpha_{0}+N\right)}\right)\right],\nonumber 
\end{align}
$\forall N\in\mathbb{N},$ as $t\rightarrow\infty$. The logarithmic
terms above arise from integrating terms of order $\tau^{j}$, with $j$ a negative integer,
in the classical expansion of the given symbol. In particular the
leading logarithmic term has coefficient $d_{0}=\frac{1}{2\pi^{2}}\partial_{z_{1}}^{\alpha_{1}}\partial_{\bar{z}_{1}}^{\alpha_{2}}b_{3+\alpha_{0}}$.

The above allows us to compute the asymptotics of both sides of the
equation \prettyref{eq:comparing coeff. eqn.} as $t\rightarrow\infty$.
In particular the right hand side of \prettyref{eq:comparing coeff. eqn.}
is seen to contain the logarithmic term 
\[
\frac{9\pi^{2}}{2}b_{3}^{4}\left(\frac{1}{2\pi^{2}}t^{-2-\frac{2}{r}}\ln t\right)^{4}
\]
in its asymptotic expansion. Such a term involving the fourth power
of a logarithm is missing from the left hand side of \prettyref{eq:comparing coeff. eqn.}.
This particularly gives $b_{3}=0$.

Using \prettyref{thm: thm about vanish coeff.}, it now follows that
$q(z,\bar{z})=c_{0}(z_{1}\bar{z}_{1})^{\frac{r}{2}-1}$ for some $c_{0}>0$.
Since $p$ has no purely holomorphic or anti-holomorphic terms in
$z_{1}$, this gives $p=\frac{c}{2}\left(z_{1}\bar{z}_{1}\right)^{\frac{r}{2}}$
for some $c>0.$ 

However, the model kernel $B_{p}$ for this potential $p=\frac{c}{2}\left(z_{1}\bar{z}_{1}\right)^{\frac{r}{2}}$
was computed in \prettyref{thm:Monomial model calc.}. Suppose $r>2$.
By \prettyref{thm:Monomial model calc.} and definition of $\tilde{B}_{p,\alpha_{0}}$
in \prettyref{eq:computation of the leading term}, 
\begin{align*}
\tilde{B}_{p,\alpha_{0}}\left(0\right) & =\frac{1}{\pi}\Gamma\left(2+\frac{2}{r}+\alpha_{0}\right)B_{p}\left(0\right)=\frac{1}{2\pi^{2}}\Gamma\left(2+\frac{2}{r}+\alpha_{0}\right)\frac{r}{\Gamma\left(\frac{2}{r}\right)}c^{\frac{2}{r}};\\
\left[\partial_{z_{1}}\tilde{B}_{p,\alpha_{0}}\right]\left(0\right) & =\left[\partial_{\overline{z}_{1}}\tilde{B}_{p,\alpha_{0}}\right]\left(0\right)=0;\\
\left[\partial_{z_{1}}\partial_{\bar{z}_{1}}\tilde{B}_{p,\alpha_{0}}\right]\left(0\right) & =\frac{1}{\pi}\Gamma\left(2+\frac{4}{r}+\alpha_{0}\right)\left[\partial_{z_{1}}\partial_{\bar{z}_{1}}B_{p}\right]\left(0\right)\\
 & =\frac{1}{2\pi^{2}}\Gamma\left(2+\frac{4}{r}+\alpha_{0}\right)\frac{r}{\Gamma\left(\frac{4}{r}\right)}c^{\frac{4}{r}}.
\end{align*}

Plugging the above into \prettyref{eq:comparing coeff. eqn.} with
$z_{1}=0$, and noting $\Delta p(0)=0$ as $r>2$, we obtain 
\begin{align*}
 & \left(\frac{r}{2\pi^{2}}\right)^{3}\frac{\Gamma\left(2+\frac{4}{r}\right)}{\Gamma\left(\frac{4}{r}\right)}c^{\frac{8}{r}}\left[\frac{\Gamma\left(2+\frac{2}{r}\right)}{\Gamma\left(\frac{2}{r}\right)}\frac{\Gamma\left(4+\frac{2}{r}\right)}{\Gamma\left(\frac{2}{r}\right)}-\frac{\Gamma\left(3+\frac{2}{r}\right)}{\Gamma\left(\frac{2}{r}\right)}\frac{\Gamma\left(3+\frac{2}{r}\right)}{\Gamma\left(\frac{2}{r}\right)}\right]\\
 & =\frac{9\pi^{2}}{2}\left[\frac{r}{2\pi^{2}}\frac{\Gamma\left(2+\frac{2}{r}\right)}{\Gamma\left(\frac{2}{r}\right)}c^{\frac{2}{r}}\right]^{4}.
\end{align*}
Using $\Gamma\left(z+1\right)=z\Gamma\left(z\right)$, the above simplifies
to the equation 
\[
\left(1+\frac{4}{r}\right)\left(2+\frac{2}{r}\right)=\frac{9}{4}\left(1+\frac{2}{r}\right)^{2}.
\]
Solving this quadratic equation yields $r=2$, a plain contradiction.
This finishes the proof. 
\end{proof}
\begin{rem}
Note in our proof above, we compared the $\left(t^{-2-\frac{2}{r}}\ln t\right)^{4}$
term on both sides of \prettyref{eq:comparing coeff. eqn.}. For that,
we only used the information of $b_{3}$, where $b_{3}$ arises in
the coefficient of the first $\ln t$ term in the asymptotics for
the model Bergman kernel (see \prettyref{eq:asymptotic of det terms}).
The authors also compared the non-logarithmic terms on two sides of
\prettyref{eq:comparing coeff. eqn.}: the $\left(t^{-2-\frac{2}{r}}\right)^{4}$
and $\left(t^{-2-\frac{2}{r}}\right)^{4}t^{-1}$ terms, whose calculations
then involve $b_{0}$ and $b_{1}.$ Nevertheless, we only got tautologies
and thus derived no contradiction. It is interesting to compare this
with the proofs of Cheng's conjecture. In dimension $2$, Fu-Wong
\cite{Fu-Wong97} used information of the logarithmic term in the
Fefferman expansion of the Bergman kernel \prettyref{eq:Fefferman expansion};
while in higher dimension, Huang and the second author \cite{Huang-Xiao21}
utilized information of the non-logarithmic term (principal singular
term) in the expansion \prettyref{eq:Fefferman expansion}. 
\end{rem}

 \bibliographystyle{siam}
\bibliography{biblio2}

\end{document}